\documentclass{article}%
\usepackage{graphicx}
\usepackage{amsmath}
\usepackage{amsfonts}
\usepackage{amssymb}
\usepackage{setspace}
\usepackage[hmargin=3.5cm,vmargin=3.5cm]{geometry}%
\providecommand{\U}[1]{\protect\rule{.1in}{.1in}}
\providecommand{\U}[1]{\protect\rule{.1in}{.1in}}
\providecommand{\U}[1]{\protect\rule{.1in}{.1in}}
\providecommand{\U}[1]{\protect\rule{.1in}{.1in}}
\providecommand{\U}[1]{\protect\rule{.1in}{.1in}}
\providecommand{\U}[1]{\protect\rule{.1in}{.1in}}
\providecommand{\U}[1]{\protect\rule{.1in}{.1in}}
\providecommand{\U}[1]{\protect\rule{.1in}{.1in}}
\providecommand{\U}[1]{\protect\rule{.1in}{.1in}}
\providecommand{\U}[1]{\protect\rule{.1in}{.1in}}
\providecommand{\U}[1]{\protect\rule{.1in}{.1in}}
\providecommand{\U}[1]{\protect\rule{.1in}{.1in}}
\providecommand{\U}[1]{\protect\rule{.1in}{.1in}}
\providecommand{\U}[1]{\protect\rule{.1in}{.1in}}
\providecommand{\U}[1]{\protect\rule{.1in}{.1in}}
\providecommand{\U}[1]{\protect\rule{.1in}{.1in}}
\providecommand{\U}[1]{\protect\rule{.1in}{.1in}}
\providecommand{\U}[1]{\protect\rule{.1in}{.1in}}
\providecommand{\U}[1]{\protect\rule{.1in}{.1in}}
\providecommand{\U}[1]{\protect\rule{.1in}{.1in}}
\providecommand{\U}[1]{\protect\rule{.1in}{.1in}}
\providecommand{\U}[1]{\protect\rule{.1in}{.1in}}
\providecommand{\U}[1]{\protect\rule{.1in}{.1in}}
\providecommand{\U}[1]{\protect\rule{.1in}{.1in}}
\providecommand{\U}[1]{\protect\rule{.1in}{.1in}}
\providecommand{\U}[1]{\protect\rule{.1in}{.1in}}
\providecommand{\U}[1]{\protect\rule{.1in}{.1in}}
\providecommand{\U}[1]{\protect\rule{.1in}{.1in}}
\providecommand{\U}[1]{\protect\rule{.1in}{.1in}}
\providecommand{\U}[1]{\protect\rule{.1in}{.1in}}
\providecommand{\U}[1]{\protect\rule{.1in}{.1in}}
\providecommand{\U}[1]{\protect\rule{.1in}{.1in}}
\providecommand{\U}[1]{\protect\rule{.1in}{.1in}}
\providecommand{\U}[1]{\protect\rule{.1in}{.1in}}
\providecommand{\U}[1]{\protect\rule{.1in}{.1in}}
\providecommand{\U}[1]{\protect\rule{.1in}{.1in}}
\providecommand{\U}[1]{\protect\rule{.1in}{.1in}}
\providecommand{\U}[1]{\protect\rule{.1in}{.1in}}
\providecommand{\U}[1]{\protect\rule{.1in}{.1in}}
\providecommand{\U}[1]{\protect\rule{.1in}{.1in}}

\newtheorem{theorem}{Theorem}
{}

{}
\newtheorem{notation}{Notation}

\newtheorem{proposition}{Proposition}
\newtheorem{remark}{Remark}

\newenvironment{proof}[1][Proof]{\textbf{#1.} }{\ \rule{0.5em}{0.5em}}

\begin{document}

\title{On a Class of Non-self-adjoint Multidimensional Periodic Schrodinger Operators }
\author{O. A. Veliev\\{\small Depart. of Math., Dogus University, Ac\i badem, Kadik\"{o}y, \ }\\{\small Istanbul, Turkey.}\ {\small e-mail: oveliev@dogus.edu.tr}}
\date{}
\maketitle

\begin{abstract}
We investigate the multidimensional Schrodinger operator $L(q)$ in
$L_{2}\left(  \mathbb{R}^{d}\right)  \ (d\geq2)$ with complex-valued periodic,
with respect to a lattice $\Omega,$ potential \ $q$ when the Fourier
coefficients $q_{\gamma}$ of $q$ with respect to the orthogonal system
$\{e^{i\left\langle \gamma x\right\rangle }:\gamma\in\Gamma\}$ vanish if
$\gamma$ belong to a half-space, where $\Gamma$ is the lattice dual to
$\Omega.$ We prove that the Bloch eigenvalues are $\mid\gamma+t\mid^{2}$ for
$\gamma\in\Gamma,$ where $t$ is a quasimomentum and find explicit formulas for
\ the Bloch functions. It implies that the Fermi surfaces of $L(q)$ and $L(0)$
are the same. The considered set of operators includes a large class of PT
symmetric operators used in the PT symmetric quantum theory.



\end{abstract}

\section{Introduction and Preliminary Facts}

We consider the Schrodinger operator $L(q)$ generated in $L_{2}(\mathbb{R}%
^{d})$ by the expression%
\begin{equation}
-\Delta\Psi+q\Psi,
\end{equation}
where $d\geq2$ and $q$ is a periodic, relative to a lattice $\Omega$,
potential belonging to the class $S$ of the complex-valued function defined as
follows. Let
\[
\Gamma:=\left\{  \gamma\in\mathbb{R}^{d}:\left\langle \gamma,\omega
\right\rangle \in2\pi\mathbb{Z}\text{, }\forall\omega\in\Omega\right\}
\]
be the lattice dual to $\Omega$, where $\mathbb{Z}$ is the set of all integers
and $\left\langle \cdot,\cdot\right\rangle $ is the inner product in
$\mathbb{R}^{d}.$ Let $v_{1},v_{2},...,v_{d}$ be any generator of the
reciprocal lattice $\Gamma,$ that is,
\[
\Gamma=\left\{  n_{1}v_{1}+n_{2}v_{2}+...+n_{d}v_{d}:\text{ }n_{1}%
\in\mathbb{Z},\text{ }n_{2}\in\mathbb{Z},...,n_{d}\in\mathbb{Z}\right\}  .
\]
Divide the lattice $\Gamma$ into three parts $\Gamma(k),$ $\Gamma(k+)$ and
$\Gamma(k-),$ where $\Gamma(k)$ is the sublatice of $\Gamma$ generated by
$\left\{  v_{1},v_{2},...,v_{d}\right\}  \backslash\left\{  v_{k}\right\}  $
and
\begin{equation}
\Gamma(k\pm)=\left\{  u\pm nv_{k}:u\in\Gamma(k),n\in\mathbb{N}\right\}
,\text{ }\mathbb{N=}\left\{  1,2,...\right\}  .
\end{equation}
Denote by $S(k\pm)$ the set of potentials $q$ whose Fourier decompositions
have the form
\begin{equation}
q(x)=\sum_{\gamma\in\Gamma(k\pm)}q_{\gamma}e^{i\left\langle \gamma
,x\right\rangle }%
\end{equation}
and the Fourier coefficients $q_{\gamma}$ satisfy the following inequality
\begin{equation}
\sum_{\gamma\in\Gamma(k\pm)}\mid q_{\gamma}\mid=M<\infty,
\end{equation}
where $q_{\gamma}=(q,e^{i\left\langle \gamma,x\right\rangle }),$ $\left(
\cdot,\cdot\right)  $ is the inner product in $L_{2}(F)$ and $F:=\mathbb{R}%
^{d}/\Omega$ is the fundamental domain (primitive cell) of the lattice
$\Omega$. Without loss of generality we assume that the measure $\mu(F)$ of
$F$ is $1$. Define $S$ by
\begin{equation}
S=\cup_{k=1}^{d}\left(  S(k+)\cup S(k-)\right)  .
\end{equation}
The operator $L(q)$ for each nonzero $q\in S$ is non-self-adjoint. However
$\left\{  L(q):q\in S\right\}  $ contains a large class $\left\{  L(q):q\in
S,\text{ }q_{\gamma}\in\mathbb{R}\text{, }\forall\gamma\in\Gamma\right\}  $ of
PT symmetric operators which are important in the PT symmetric quantum theory
(see e.g. [1]).

Let $L_{t}(q)$ be the operator generated in $L_{2}(F)$ by (1) and the
quasiperiodic conditions%
\begin{equation}
\Psi(x+\omega)=e^{i\left\langle t,\omega\right\rangle }\Psi(x),\ \forall
\omega\in\Omega,
\end{equation}
where $t\in F^{\star}:=\mathbb{R}^{d}/\Gamma$. It is well-known that the
spectrum of $L_{t}(q)$ consists of the eigenvalues $\Lambda_{1}(t),$
$\Lambda_{2}(t),....$ which are called Bloch eigenvalues of $L(q)$. The
eigenfunction $\Psi_{N,t}(x)$ of $L_{t}(q)$ corresponding to the eigenvalue
$\Lambda_{N}(t)$ is known as Bloch function of $L(q)$:%
\begin{equation}
L_{t}(q)\Psi_{N,t}(x)=\Lambda_{N}(t)\Psi_{N,t}(x).
\end{equation}
In the case $q=0$ the eigenvalues and normalized eigenfunctions of $L_{t}(q)$
are $\mid\gamma+t\mid^{2}$ and $e^{i\left\langle \gamma+t,x\right\rangle }$
for $\gamma\in\Gamma$:%
\begin{equation}
\text{ }L_{t}(0)e^{i\left\langle \gamma+t,x\right\rangle }=\mid\gamma
+t\mid^{2}e^{i\left\langle \gamma+t,x\right\rangle }.
\end{equation}

The one dimensional case were considered in detail. Gasymov [5] proved the
following results for the operator $L(q):=-\frac{d^{2}}{dx^{2}}+q$ with the
potential $q$ of the form \
\begin{equation}
q(x)=%
%TCIMACRO{\tsum \limits_{n=1}^{\infty}}%
%BeginExpansion
{\textstyle\sum\limits_{n=1}^{\infty}}
%EndExpansion
q_{n}e^{inx},\text{ }%
\end{equation}
when $%
%TCIMACRO{\tsum }%
%BeginExpansion
{\textstyle\sum}
%EndExpansion
\mid q_{n}\mid<\infty.$

\textbf{Result 1:} The spectrum $\sigma(L(q))$ of $L(q)$ is $[0,\infty).$
There may be the spectral singularities on the spectrum which must coincide
with numbers of the form $(\frac{n}{2})^{2}.$

\textbf{Result 2:} The equation
\begin{equation}
-y^{^{\prime\prime}}(x)+q(x)y(x)=\mu^{2}y(x)
\end{equation}
has the Floquet solution of the form
\[
f(x,\mu)=e^{i\mu x}(1+%
%TCIMACRO{\tsum \limits_{n=1}^{\infty}}%
%BeginExpansion
{\textstyle\sum\limits_{n=1}^{\infty}}
%EndExpansion
\frac{1}{n+2\mu}%
%TCIMACRO{\tsum \limits_{\alpha=n}^{\infty}}%
%BeginExpansion
{\textstyle\sum\limits_{\alpha=n}^{\infty}}
%EndExpansion
v_{n,\alpha}e^{i\alpha x}),
\]
where the following series converge
\[%
%TCIMACRO{\tsum \limits_{n=1}^{\infty}}%
%BeginExpansion
{\textstyle\sum\limits_{n=1}^{\infty}}
%EndExpansion
\frac{1}{n}%
%TCIMACRO{\tsum \limits_{\alpha=n+1}^{\infty}}%
%BeginExpansion
{\textstyle\sum\limits_{\alpha=n+1}^{\infty}}
%EndExpansion
\alpha(\alpha-n)\mid v_{n,\alpha}\mid,\text{ }%
%TCIMACRO{\tsum \limits_{n=1}^{\infty}}%
%BeginExpansion
{\textstyle\sum\limits_{n=1}^{\infty}}
%EndExpansion
n\mid v_{n,\alpha}\mid.
\]

\textbf{Result 3:} By the Floquet solutions a spectral expansion was constructed.

\textbf{Result 4:} It was shown that the Wronskian of the Floquet solutions%
\[
f_{n}(x):=\lim_{\mu\rightarrow\frac{n}{2}}(n-2\mu)f(x,-\mu)
\]
and $f(x,\frac{n}{2})$ is equal to zero and $f_{n}(x)=s_{n}f(x,\frac{n}{2}).$
It was proved that from $\{s_{n}\}$ one can effectively reconstruct
$\{q_{n}\}$.

Guillemin and Uribe [7,8] investigated the boundary value problem (bvp)
generated on $[0,2\pi]$ by (10) and the periodic boundary conditions when
$q\in Q_{2}^{+},$ where $Q_{2}^{+}$ is the set of $q\in L_{2}[0,2\pi]$ of the
form (9). It was proved that the eigenvalues of this bvp are $n^{2}$ for
$n\in\mathbb{Z}$ and the root functions (eigenfunction and associated
functions) were studied. Moreover, the inverse method applied to Hill's
equation was developed in the class of Hardy potentials and its applications
to the N-soliton solutions of KdV was considered. For $L(q)$ with the
potential $q\in Q_{2}^{+}$ the inverse spectral problem was investigated by
Pastur and Tkachenko [10] and the alternative proofs of the equality
$\sigma(L(q))=[0,\infty)$ were provided by Shin [11], Carlson [2] and
Christiansen [3]. In the case $q(x)=Ae^{2\pi irx},$ where $A\in\mathbb{C}$ and
$r\in\mathbb{Z}$, the periodic and antiperiodic bvp was investigated in detail
by Kerimov [9]. Several new and interesting from the point of view of
physicists observations about $L(q)$ were made by T. Curtright, L. Mezincescu [4].

In the paper [12], we proved that if
\begin{equation}
q\in L_{1}[0,1],\text{ }q(x+1)=q(x),\text{ }q_{n}=0,\text{ }\forall
n=0,-1,-2,...,
\end{equation}
where $q_{n}=(q,e^{i2\pi nx})$ and $(\cdot,\cdot)$ is the inner product in
$L_{2}[0,1],$ then $\sigma(L(q))=[0,\infty)$ and $\sigma(L_{t}(q))=\{(2\pi
n+t)^{2}:n\in\mathbb{Z}\}$ for all $t\in\mathbb{C},$ where $L_{t}(q)$ is the
operator generated in $L_{2}[0,1]$ by the bvp
\begin{equation}
-y^{^{\prime\prime}}(x)+q(x)y(x)=\lambda y(x),\text{ }y(1)=e^{it}y(0),\text{
}y^{^{\prime}}(1)=e^{it}y^{^{\prime}}(0).
\end{equation}
Moreover, we proved that if $t\neq0,\pi$, then $(2\pi n+t)^{2}$ is a simple
eigenvalue of $L_{t}(q)$ and find explicit formulas for the\ corresponding
eigenfunctions. Finally, we consider the inverse problem for the general case (11).

As far as I am concerned for the multidimensional Schrodinger operator with a
potential from some subset of $S$ there exists only one paper [6]. In [6]
Gasimov investigated the three dimensional operator $L(q)$ with periodic,
relative to $2\pi\mathbb{Z}^{3},$ potential of the form
\begin{equation}
q(x)=\sum_{\gamma\in Z^{+}}q_{\gamma}e^{i\left\langle \gamma,x\right\rangle
},\text{ }\sum_{\gamma\in Z^{+}}\mid q_{\gamma}\mid<\infty,
\end{equation}
where $Z^{+}=\left\{  (m_{1},m_{2},m_{3})\in\mathbb{Z}^{3}:m_{1}+m_{2}%
+m_{3}\geq1,\text{ }m_{j}\geq0,\forall j=1,2,3\right\}  $ is a part of
$\mathbb{Z}^{3}$ lying in the first octant of $\mathbb{R}^{3}.$ He announces a
formula for the resolvent kernel of $\left(  L(q)-k^{2}\right)  ^{-1}$ with
$\operatorname{Im}k\neq0$, which indicates that $\left(  L-k^{2}\right)
^{-1}$ is bounded for such $k$. He also announces that the spectrum of $L$ is
$[0,\infty)$, and gives a Plancherel theorem, where the Fourier transform is
given in terms of certain solutions to $-\Delta u+qu=k^{2}u$.

In this papers, by combining the methods of my papers [12-16], we investigate
the periodic multidimensional Schrodinger operator $L(q)$ of arbitrary
dimension $d$ and arbitrary lattice $\Omega$ when the potential $q$ belongs to
the large than (13) class $S$. The main results of this paper are formulated
in Section 2, where we prove that if $q\in S$ then for all $t\in F^{\star}$
the eigenvalues of $L_{t}(q)$ consist of the numbers $\left\vert
\gamma+t\right\vert ^{2}$, for $\ \gamma\in\Gamma,$ that is, the Bloch
eigenvalues of $L(q)$ for $q\in S$ coincides with the Bloch eigenvalues of the
free operator $L(0).$ It implies that the isoenergetic surfaces of $L(q)$ and
$L(0)$ are the same. Moreover, we find explicit formulas for the\ Bloch
functions. At the end of Section 2, we prove that in the two and three
dimensional cases the main results continue to hold if (4) is replaced by
$q\in L_{2}(F).$ In Section 3 using the results of Section 2 and the approach
of the papers [7,8] we investigate the multiplicity of the Bloch eigenvalue
and consider the necessary and sufficient conditions on the potential which
provide some root functions to be eigenfunctions. Besides, we investigate in
detail the root function of the bvp (12) for $t=0,\pi$ when the potential $q$
satisfies (11).

\section{ Main Results}

First let us formulate the main results. For this we introduce the following
notations and recall some well known facts. For $b\in\Gamma$ the hyperplanes
$\left\{  x\in\mathbb{R}^{d}:\left\vert x\right\vert =\left\vert
x+b\right\vert \right\}  $ are called diffraction hyperplanes. The number
$\mid\gamma+t\mid^{2}$ is a simple eigenvalue of $L_{t}(0)$ if and only if
$\gamma+t$ does not belong to any diffraction hyperplane, that is,
\begin{equation}
\mid\gamma+t\mid\neq\mid\gamma+b+t\mid,\text{ }\forall b\in\Gamma,\text{
}b\neq0.
\end{equation}

A number $\lambda$ is a multiple eigenvalue of $L_{t}(0)$ of multiplicity $m$
if and only if there exist $m$ different vectors $b_{1},b_{2},...,b_{m}$ of
the lattice $\Gamma$ such that
\begin{equation}
\lambda=\mid b_{1}+t\mid^{2}=\mid b_{2}+t\mid^{2}=...=\mid b_{m}+t\mid^{2}.
\end{equation}
By the definitions of $v_{k}$ and $\Gamma(k)$ (see (2)) for each $b_{j}$ there
exists $p_{j}\in\mathbb{Z}$ such that
\begin{equation}
b_{j}\in\Gamma(k,p_{j}):=\left\{  p_{j}v_{k}+a:a\in\Gamma(k)\right\}  .
\end{equation}
Let us enumerate the vectors $b_{j}$ so that $p_{1}\geq p_{2}\geq...$ Then
there exists $s$ such that
\begin{equation}
p_{1}=p_{2}=...=p_{s}>p_{s+1}\geq p_{s+2}\geq...
\end{equation}

Finally recall that the isoenergetic surfaces of the operators $L(q)$ and
$L(0)$ corresponding to the energy $\rho^{2}$ are respectively the sets
\[
I_{\rho}(q)=\{t\in F^{\ast}:\exists N,\Lambda_{N}(t)=\rho^{2}\}\text{ }%
\And\text{ }I_{\rho}(0)=\{t\in F^{\ast}:\exists\gamma\in\Gamma,\mid
\gamma+t\mid=\rho\}.
\]
The isoenergetic surface $I_{\rho}(0)$ is the translations of the sphere
$\{\mid x\mid=\rho\}$ by the vectors $\gamma\in\Gamma$ to the fundamental
domain $F^{\ast}$ of the reciprocal lattice $\Gamma.$

\begin{theorem}
(\textbf{Main Results}). $(a)$ If $q\in S$, then for any $t\in F^{\ast}$, the
set of eigenvalues of $L_{t}(q)$ is $\left\{  \mid\gamma+t\mid^{2}:\gamma
\in\Gamma\text{ }\right\}  ,$ that is,
\[
\sigma(L_{t}(q))=\sigma(L_{t}(0)),\text{ }\forall t\in F^{\ast}\And
\sigma(L(q))=%
%TCIMACRO{\tbigcup \limits_{t\in F^{\ast}}}%
%BeginExpansion
{\textstyle\bigcup\limits_{t\in F^{\ast}}}
%EndExpansion
\sigma(L_{t}(q))=[0,\infty).
\]

$(b)$ For any $q\in S$ and $\rho\in\lbrack0,\infty)$ the isoenergetic surface
$I_{\rho}(q)$ of $L(q)$ coincides with the isoenergetic surface $I_{\rho}(0)$
of the free operator $L(0).$

$(c)$ Let $\mid\gamma+t\mid^{2}$ be a simple eigenvalue of $L_{t}(0)$. If
$q\in S(k+),$ then there exists only one eigenfunction $\Psi_{\gamma+t}(x)$ of
$L_{t}(q)$ corresponding to $\mid\gamma+t\mid^{2}.$ It can be normalized by
\begin{equation}
(\Psi_{\gamma+t},e^{i\left\langle \gamma+t,x\right\rangle })=1
\end{equation}
and satisfies (19) and (47) (see below and Theorem 4).

$(d)$ Let $\lambda$ be the eigenvalue of $L_{t}(0)$ \ of multiplicity $m.$ If
$q\in S(k+)$, then the functions $\Psi_{b_{j}+t}(x)$ for $j=1,2,...,s$ defined
in (21) are linearly independent eigenfunctions of $L_{t}(q)$ corresponding to
$\lambda$, where $b_{j}$ and $s$ are defined in (15)-(17).

The theorem continues to hold if $S(k+)$ and $\Gamma(k+)$ are replaced by
$S(k-)$ and $\Gamma(k-).$\textit{ } In the case $d=2$ and $d=3$ the main
results continue to hold if the condition (4) is replaced by $q\in L_{2}(F).$
\end{theorem}

\begin{proof}
The proof of $(a)$ follows from Theorem 2 and Theorem 3, where the relations
$\sigma(L_{t}(q))\subset\sigma(L_{t}(0))$ and $\sigma(L_{t}(0))\subset
\sigma(L_{t}(q))$ are proved respectively. $(b)$ follows from $(a).$ In
Theorem 3 we prove that if (15)-(17) holds then the function $\Psi_{b_{j}%
+t}(x)$ defined in (21) is an eigenfunction. Since it is clear that the
functions $\Psi_{b_{j}+t}(x)$ for $j=1,2,...,s$ are linearly independent we
get the prove of $(d),$ by proving Theorem 3. The theorems 3 and 4 imply
$(c),$ because only one eigenfunction may provide (19) and (47). We prove the
theorems for $q\in S(k+).$ The proof of the case $q\in S(k-)$ is the same. By
(5) they are enough for the proof of the statements for $q\in S.$ The last
statement is proved in Theorem 5
\end{proof}

Let $\mid\gamma+t\mid^{2}$ be the simple eigenvalue of \ $L_{t}(0),$ which
means that (14) holds. Introduce the function $\Psi_{\gamma+t}(x)$ defined by
\begin{equation}
\Psi_{\gamma+t}(x)=e^{i\left\langle \gamma+t,x\right\rangle }+A(\gamma
)e^{i\left\langle \gamma+t,x\right\rangle }+\left(  A(\gamma)\right)
^{2}e^{i\left\langle \gamma+t,x\right\rangle }+...,
\end{equation}
where $A(\gamma)$ is the linear transformation taking $e^{i\left\langle
b+t,x\right\rangle }$ for $b\in\left(  \gamma+\Gamma(k+)\right)  \cup\left\{
\gamma\right\}  $ to
\begin{equation}
A(\gamma)e^{i\left\langle b+t,x\right\rangle }=\sum_{\gamma_{1}\in\Gamma
(k+)}\frac{q_{\gamma_{1}}e^{i\left\langle b+\gamma_{1}+t,x\right\rangle }%
}{\mid\gamma+t\mid^{2}-\mid b+\gamma_{1}+t\mid^{2}}.
\end{equation}
In the proof of Theorem 3 we will prove that the series (19) converges to some
element of $L_{2}(F).$ Now we only note that if (14) holds then the
denominator in (20) are not zero and $(A(\gamma))^{n}$ for all $n=1,2,...,$
are defined.

Similarly in case (15)-(17), using the definition of $\Gamma(k+),$ one can
readily see that the transformations $\left(  A(b_{j})\right)  ^{n}$ for
$j=1,2,...,s$ and $n=1,2,...,$ are defined on $e^{i\left\langle
b+t,x\right\rangle }$ for $b\in\left(  b_{j}+\Gamma(k+)\right)  \cup\left\{
b_{j}\right\}  .$ Introduce the function $\Psi_{b_{j}+t}(x)$ defined by
\textit{ }%
\begin{equation}
\text{ }\Psi_{b_{j}+t}(x)=e^{i\left\langle b_{j}+t,x\right\rangle }+\left(
A(b_{j})\right)  e^{i\left\langle b_{j}+t,x\right\rangle }+\left(
A(b_{j})\right)  ^{2}e^{i\left\langle b_{j}+t,x\right\rangle }+...
\end{equation}

\begin{remark}
One can readily see that the transformation $A(b)$ can be defined by
\begin{equation}
A(b)f=(\mid b+t\mid^{2}I+\Delta)^{-1}qf
\end{equation}
in some subspaces. It is clear that if (14) holds then $(\mid\gamma+t\mid
^{2}I+\Delta)^{-1}$ is well-defined in the subspace $E(\gamma)$ generated by
the orthonormal system $\left\{  e^{i\left\langle b+t,x\right\rangle }%
:b\in\Gamma\backslash\gamma\right\}  .$ On the other hand, it follows from (2)
and definition of $S(k+)$ that if $q\in S(k+)$ then $q^{n}e^{i\left\langle
\gamma+t,x\right\rangle }\in E(\gamma)$ for $n=1,2,...$. Therefore $\left(
A(\gamma)\right)  ^{n}e^{i\left\langle \gamma+t,x\right\rangle }$ exist for
$n=1,2,....$Similarly, in case (15)-(17), $(\mid b_{j}+t\mid^{2}I+\Delta
)^{-1}$ is defined in the subspace $E(b_{1},b_{2},...,b_{m})$ generated by the
orthonormal system $\left\{  e^{i\left\langle b+t,x\right\rangle }:b\in
\Gamma\backslash\left\{  b_{1},b_{2},...,b_{m}\right\}  \right\}  .$ On the
other hand, by (2) if $q\in S(k+)$ and $j=1,2,...,s$ then $q^{n}%
e^{i\left\langle b_{j}+t,x\right\rangle }\in E(b_{1},b_{2},...,b_{m})$ for
$n=1,2,...$. Therefore $\left(  A(b_{j})\right)  ^{n}e^{i\left\langle
b_{j}+t,x\right\rangle }$ exist for $n=1,2,....,$ and $j=1,2,...,s$.
\end{remark}

To consider the Bloch eigenvalues $\Lambda_{N}(t)$ and Bloch functions
$\Psi_{N,t}$ we use the following iteration of the formula
\begin{equation}
(\Lambda_{N}(t)-\mid\gamma+t\mid^{2})(\Psi_{N,t},e^{i\left\langle
\gamma+t,x\right\rangle })=(q\Psi_{N,t},e^{i\left\langle \gamma
+t,x\right\rangle })
\end{equation}
which is obtained from\ (7) by multiplying by $e^{i\left\langle \gamma
+t,x\right\rangle }$ and using (8). If
\begin{equation}
\Lambda_{N}(t)\neq\mid\gamma+t\mid^{2},\text{ }\forall\gamma\in\Gamma,
\end{equation}
then (23) can be iterated as follows. Using \ the expansion (3) of $q\in
S(k+)$ in (23), we get
\begin{equation}
(\Lambda_{N}(t)-\mid\gamma+t\mid^{2})(\Psi_{N,t},e^{i\left\langle
\gamma+t,x\right\rangle })=\sum_{\gamma_{1}\in\Gamma(k+)}q_{\gamma_{1}}%
(\Psi_{N,t},e^{i\left\langle \gamma-\gamma_{1}+t,x\right\rangle }).
\end{equation}
On the other hand, in (23) replacing $\gamma$ by $\gamma-\gamma_{1}$ and
taking into account (24) we obtain
\begin{equation}
(\Psi_{N,t},e^{i\left\langle \gamma-\gamma_{1}+t,x\right\rangle }%
)=\dfrac{(q\Psi_{N,t},e^{i\left\langle \gamma-\gamma_{1}+t,x\right\rangle }%
)}{\Lambda_{N}(t)-\mid\gamma-\gamma_{1}+t\mid^{2}}.
\end{equation}
Now, using (26) in (25) we get
\begin{equation}
(\Lambda_{N}(t)-\mid\gamma+t\mid^{2})(\Psi_{N,t},e^{i\left\langle
\gamma+t,x\right\rangle })=\sum_{\gamma_{1}}\dfrac{q_{\gamma_{1}}(q\Psi
_{N,t},e^{i\left\langle \gamma-\gamma_{1}+t,x\right\rangle })}{\Lambda
_{N}(t)-\mid\gamma-\gamma_{1}+t\mid^{2}}.
\end{equation}
Repeating this process $m$ times we obtain%

\begin{equation}
(\Lambda_{N}(t)-\mid\gamma+t\mid^{2})(\Psi_{N,t},e^{i\left\langle
\gamma+t,x\right\rangle })=\sum_{\gamma_{1},\gamma_{2},...,\gamma_{m}}%
\frac{q_{\gamma_{1}}q_{\gamma_{2}}...q_{\gamma_{m}}(q\Psi_{N,t}%
,e^{i\left\langle \gamma+t-\gamma(m),x\right\rangle })}{%
%TCIMACRO{\tprod \limits_{s=1,2,...,m}}%
%BeginExpansion
{\textstyle\prod\limits_{s=1,2,...,m}}
%EndExpansion
[\Lambda_{N}(t)-\left\vert \gamma+t-\gamma(s)\right\vert ^{2}]},
\end{equation}
where $\gamma(j)=:\gamma_{1}+\gamma_{2}+\cdots+\gamma_{j}$ for $j=1,2,....$
and the summations in (27) and (28) are taken under the conditions $\gamma
_{1}\in\Gamma(k+)$ and $\gamma_{1},\gamma_{2},...,\gamma_{m}\in\Gamma(k+)$ respectively.

To estimate the right-hand side of (28) we use the following simple proposition.

\begin{proposition}
For each $k\in\left\{  1,2,...d\right\}  $ there exists a positive constant
$c(k)$ such that if $\gamma_{j}\in\Gamma(k+)$ for $j=1,2,....,$ then
\begin{equation}
\left\vert \gamma_{1}+\gamma_{2}+\cdots+\gamma_{s}\right\vert \geq c(k)s
\end{equation}
for all $s\in\mathbb{N}$. The proposition continues to hold if $\Gamma(k+)$ is
replaced by $\Gamma(k-).$
\end{proposition}

\begin{proof}
We prove the proposition for $\Gamma(k+).$ The proof for $\Gamma(k-)$ is the
same. Let $\ P(k)$ be the hyperplane spanned by $v_{1},v_{2},...,v_{k-1}%
,v_{k+1},...,v_{d}.$ Since $v_{k}\notin P(k),$ it has the orthogonal
decomposition%
\begin{equation}
v_{k}=u_{k}+h_{k},\text{ }u_{k}\in P(k),\text{ }h_{k}\perp P(k)=0,\text{
}h_{k}\neq0.
\end{equation}
On the other hand, by (2), if $\gamma_{j}\in\Gamma(k+)$ then $\gamma_{j}%
=a_{j}+n_{j}v_{k}$, where $a_{j}\in\Gamma(k)\subset P(k)$ and $n_{j}%
\in\mathbb{N}$. Therefore there exist $u\in\Gamma(k)\subset P(k)$ and $w\in
P(k)$ such that
\begin{equation}%
%TCIMACRO{\tsum \limits_{j=1}^{s}}%
%BeginExpansion
{\textstyle\sum\limits_{j=1}^{s}}
%EndExpansion
\gamma_{j}=u+\left(
%TCIMACRO{\tsum \limits_{j=1}^{s}}%
%BeginExpansion
{\textstyle\sum\limits_{j=1}^{s}}
%EndExpansion
n_{j}\right)  v_{k}=w+\left(
%TCIMACRO{\tsum \limits_{j=1}^{s}}%
%BeginExpansion
{\textstyle\sum\limits_{j=1}^{s}}
%EndExpansion
n_{j}\right)  h_{k},
\end{equation}
where $\left\langle u,h_{k}\right\rangle =0$ (see (30)), $\left\langle
w,h_{k}\right\rangle =0$ and $n_{1}+n_{2}+\cdots+n_{s}\geq s$. Thus using (31)
and Pythagorean theorem we see that (29) holds for $c(k)=\left\Vert
h_{k}\right\Vert $
\end{proof}

Now we are ready to prove the following

\begin{theorem}
If $q\in S$ and $t\in F^{\ast}$ then for every eigenvalue $\Lambda_{N}(t)$ of
$L_{t}(q)$ there exists $\gamma\in\Gamma$ such that $\Lambda_{N}(t)=\left\vert
\gamma+t\right\vert ^{2},$ that is, $\sigma(L_{t}(q))\subset\sigma\left(
L_{t}(0)\right)  $ for all $t\in F^{\ast}.$
\end{theorem}

\begin{proof}
By (5) it is enough to prove the theorem for $q\in\left(  S(k+)\cup
S(k-)\right)  .$ We prove it for $q\in S(k+).$ The proof of the case $q\in
S(k-)$ is the same. Suppose, to the contrary that, there exists $N$ such that
(24) holds. Then there exists a positive number $c$ such that
\begin{equation}
\left\vert \Lambda_{N}(t)-\left\vert \gamma+t\right\vert ^{2}\right\vert
>c,\text{ }\forall\gamma\in\Gamma.
\end{equation}
Now using (4), (29) and (32) let us estimate the right-hand side of (28). It
immediately follows from (4) that%
\begin{equation}
\sum_{\gamma_{1},\gamma_{2},...,\gamma_{m}}\left\vert q_{\gamma_{1}}%
q_{\gamma_{2}}...q_{\gamma_{m}}\right\vert \leq M^{m}.
\end{equation}
On the other hand, by (29) we have
\begin{equation}
\left\vert \gamma+t\pm\gamma(s)\right\vert ^{2}\geq\left(  c(k)s-\left\vert
\gamma+t\right\vert \right)  ^{2}%
\end{equation}
and
\begin{equation}
\left\vert \Lambda_{N}(t)-\left\vert \gamma+t-\gamma(s)\right\vert
^{2}\right\vert \geq\left(  c(k)s-\left\vert \gamma+t\right\vert \right)
^{2}-\left\vert \Lambda_{N}(t)\right\vert .
\end{equation}
Now using (32), (33) and (35) one can readily see that the right side of (28)
approaches $0$ as $m\rightarrow\infty.$ Therefore in (28) letting $m$ tend to
$\infty$ and then using (24) we obtain%
\begin{equation}
(\Lambda_{N}(t)-\mid\gamma+t\mid^{2})(\Psi_{N,t},e^{i\left\langle
\gamma+t,x\right\rangle })=0\text{ }\And(\Psi_{N,t},e^{i\left\langle
\gamma+t,x\right\rangle })=0
\end{equation}
for all $\gamma\in\Gamma.$ The last equality of (36) is a contradiction, since
$\left\{  e^{i\left\langle \gamma+t,x\right\rangle }:\gamma\in\Gamma\right\}
$ is an orthonormal basis in $L_{2}(F)$ and $\left\Vert \Psi_{N,t}\right\Vert
\neq0.$ The theorem is proved
\end{proof}

Now we prove that $\sigma(L_{t}(0))\subset\sigma\left(  L_{t}(q)\right)  $ and
consider the Bloch function .

\begin{theorem}
Let $q\in S$ and $t\in F^{\ast}.$ Then for all $\gamma\in\Gamma$ the numbers
$\mid\gamma+t\mid^{2}$ are the eigenvalues of $L_{t}(q).$ In cases (14) and
(15)-(17) the functions defined in (19) and (21) for $j=1,2,...,s$ are the
eigenfunctions corresponding to $\mid\gamma+t\mid^{2}$ and $\lambda$ respectively.
\end{theorem}

\begin{proof}
First let us consider the case (14). It readily follows from (20) that
\begin{equation}
\left(  A(\gamma)\right)  ^{n}e^{i\left\langle \gamma+t,x\right\rangle }%
=\sum_{\gamma_{1},\gamma_{2},...,\gamma_{m}}\frac{q_{\gamma_{1}}q_{\gamma_{2}%
}...q_{\gamma_{n}}e^{i\left\langle \gamma+t+\gamma(n),x\right\rangle }}{%
%TCIMACRO{\tprod \limits_{s=1,2,...,n}}%
%BeginExpansion
{\textstyle\prod\limits_{s=1,2,...,n}}
%EndExpansion
\left(  \mid\gamma+t\mid^{2}-\left\vert \gamma+t+\gamma(s)\right\vert
^{2}\right)  }.
\end{equation}
Using (33) and (34) in (37) one can readily see that there exists a constant
$c$ such that
\begin{equation}
\left\vert \left(  A(\gamma)\right)  ^{n}e^{i\left\langle \gamma
+t,x\right\rangle }\right\vert <\frac{c}{2^{n}},\text{ }\forall n=1,2,...
\end{equation}
Therefore the series in the right hand side of (19)\ converges to some
element, denoted by $\Psi_{\gamma+t},$ of $L_{2}(F)$ and $\Psi_{\gamma+t}$
satisfies (6). Moreover, $\Delta\Psi_{\gamma+t}$ $\in L_{2}(F)$ and
\begin{equation}
A(\gamma)\Psi_{\gamma+t}=\Psi_{\gamma+t}-e^{i\left\langle \gamma
+t,x\right\rangle }).
\end{equation}
It with (22) and (8) implies that
\begin{equation}
q\Psi_{\gamma+t}=(\Delta+\mid\gamma+t\mid^{2}I)(\Psi_{\gamma+t}%
-e^{i\left\langle \gamma+t,x\right\rangle })=\Delta\Psi_{\gamma+t}+\mid
\gamma+t\mid^{2}\Psi_{\gamma+t},
\end{equation}
that is, $\Psi_{\gamma+t}$ is an eigenfunction of $L_{t}(q)$ corresponding to
$\mid\gamma+t\mid^{2}.$

Now we consider the case (15)-(17). Using (17) and (2) and arguing as in the
proof of (38) we see that in (38) one can replace $\gamma$ by $b_{j}$ for
$j=1,2,...,s.$ Therefore repeating the proof of (39) and (40) one can easily
verify that the functions $\Psi_{b_{j}+t}(x)$\ defined by (21) for
$j=1,2,...,s$ are the eigenfunctions corresponding to the eigenvalue
$\lambda=\mid b_{j}+t\mid^{2}$
\end{proof}

Now we consider the Fourier decomposition of the Bloch functions. Let
$\Psi_{\gamma+t}(x)$ be an arbitrary eigenfunction of $L_{t}(q)$ corresponding
to the the simple eigenvalue $\mid\gamma+t\mid^{2}$ of $L_{t}(0),$ that is,
(14) holds. Since $\left\{  e^{i\left\langle \gamma+\delta+t,x\right\rangle
}:\delta\in\Gamma\right\}  $ is an orthonormal basis we have
\begin{equation}
\Psi_{\gamma+t}(x)=(\Psi_{\gamma+t},e^{i\left\langle \gamma+t,x\right\rangle
})e^{i\left\langle \gamma+t,x\right\rangle }+\sum_{\delta\in\Gamma
\backslash\left\{  0\right\}  }(\Psi_{\gamma+t},e^{i\left\langle \gamma
+\delta+t,x\right\rangle })e^{i\left\langle \gamma+\delta+t,x\right\rangle }.
\end{equation}
To find $(\Psi_{\gamma+t},e^{i\left\langle \gamma+\delta+t,x\right\rangle })$
for $\delta\neq0,$ we use the formula
\begin{equation}
(\Psi_{\gamma+t},e^{i\left\langle \gamma+\delta+t,x\right\rangle })=\frac
{1}{d(\gamma,\delta)}\sum_{\gamma_{1}\in\Gamma(k+)}q_{\gamma_{1}}(\Psi
_{\gamma+t},e^{i\left\langle \gamma+\delta-\gamma_{1}+t,x\right\rangle })
\end{equation}
obtained from (25) by replacing $\Lambda_{N}(t),\Psi_{N,t}$ and $\gamma$
by$\mid\gamma+t\mid^{2},$ $\Psi_{\gamma+t}$ and $\gamma+\delta$ respectively,
where
\begin{equation}
d(\gamma,\delta)=\mid\gamma+t\mid^{2}-\mid\gamma+\delta+t\mid^{2}\neq0,\text{
}\forall\delta\neq0.
\end{equation}
Now iterating (42) we obtain the Fourier decomposition of $\Psi_{\gamma+t}$ in
the next theorem. Note that in the iteration we take into account the following.

\begin{remark}
By definition of $\Gamma(k)$, for $\delta\in\Gamma\backslash\left\{
0\right\}  $ there exist $a\in\Gamma(k)$ and $p\in\mathbb{Z}$, such that
\begin{equation}
\delta=a+pv_{k}\in\Gamma(k,p),
\end{equation}
where $\Gamma(k,p)$ is defined in (16). Let us consider the cases: $p\leq0$
and $p>0.$

\textbf{Case 1}: Let $p\leq0.$ If $\gamma_{1}\in\Gamma(k+),\gamma_{2}\in
\Gamma(k+),...,$ then $\gamma_{1}=a_{1}+p_{1}v_{k},$ $\gamma_{2}=a_{2}%
+p_{2}v_{k},...,$ where $a_{1}\in\Gamma(k),$ $a_{2}\in\Gamma(k),...$ and
$p_{1}>0,$ $p_{2}>0,...$ Therefore by (44) we have%
\[
\delta-\gamma(j)=u_{j}+s_{j}v_{k},s_{j}<0,\forall j=1,2,...,
\]
where $u_{j}\in\Gamma(k)$ and $\gamma(j)$ is defined in (28). It with (43)
implies that
\begin{equation}
\delta-\gamma(j)\neq0\text{ }\And d(\gamma,\delta-\gamma(j))\neq0,\text{
}\forall j=1,2,...\text{.}%
\end{equation}

\textbf{Case 2}: Let $p>0$ and $\gamma_{1}\in\Gamma(k+),\gamma_{2}\in
\Gamma(k+),....$ Then arguing as in \textbf{Case 1 }we obtain that
\begin{equation}
\left(  \delta-\gamma(p)\right)  \in\Gamma(k,s)\text{ }\And\text{ }s\leq0.
\end{equation}

\end{remark}

\begin{theorem}
If $\Psi_{\gamma+t}(x)$ is an eigenfunction of $L_{t}(q)$ corresponding to the
simple eigenvalue $\mid\gamma+t\mid^{2}$ of $L_{t}(0),$ then it can be
normalized by (18) and satisfies \textit{ }%
\begin{align}
\text{ }\Psi_{\gamma+t}(x)  &  =e^{i\left\langle \gamma+t,x\right\rangle
}+\sum\limits_{\delta\in\Gamma(k+)}c(\gamma,\delta)e^{i\left\langle
\gamma+\delta+t,x\right\rangle }=\nonumber\\
&  e^{i\left\langle \gamma+t,x\right\rangle }+\sum_{p\in\mathbb{N}}%
\sum\limits_{\delta\in\Gamma(k,p)}c(\gamma,\delta)e^{i\left\langle
\gamma+\delta+t,x\right\rangle },
\end{align}
where $c(\gamma,\delta)$ for $\delta\in\Gamma(k,p)$ and $p\in\mathbb{N}$ is
defined by
\begin{equation}
c(\gamma,\delta)=\frac{1}{d(\gamma,\delta)}\left(  q_{\delta}+\sum_{j=1}%
^{p-1}\sum_{\gamma_{1},\gamma_{2},...,\gamma_{j}}\frac{q_{\gamma_{1}}%
q_{\gamma_{2}}...q_{\gamma_{j}}q_{\delta-\gamma(j)}}{d(\gamma,\delta
-\gamma_{1})d(\gamma,\delta-\gamma(2))...d(\gamma,\delta-\gamma(j))}\right)
\end{equation}
and the summations is taken under conditions
\begin{equation}
\left\{  \gamma_{1},\gamma_{2},...,\gamma_{j},\delta-\gamma(j)\right\}
\in\Gamma(k+),\forall j=1,2,...,p-1
\end{equation}
Moreover, the right-hand sides of (19) and (47) are the same.
\end{theorem}

\begin{proof}
First let us consider \textbf{Case }1 of Remark 2. Then $\delta\in
\Gamma\backslash\left\{  0\right\}  $ has the form (44) and $p\leq0,$ that is,
$\delta\in\Gamma(k,p)$ and $p\leq0.$ In this case iterating (42) $m$ times and
taking into account (45) we obtain
\begin{equation}
(\Psi_{\gamma+t},e^{i\left\langle \gamma+\delta+t,x\right\rangle }%
)=\sum_{\gamma_{1},\gamma_{2},...,\gamma_{m}}\left(  \frac{q_{\gamma_{1}%
}q_{\gamma_{2}}...q_{\gamma_{m+1}}(\Psi_{\gamma+t},e^{i\left\langle
\gamma+\delta-\gamma(m+1)+t,x\right\rangle })}{d(\gamma,\delta)d(\gamma
,\delta-\gamma_{1})...d(\gamma,\delta-\gamma(m))}\right)  .
\end{equation}
Moreover, repeating the arguments used in the proof of the statement that the
right hand side of (28) approaches zero (see the proof of Theorem 2) we obtain
that the right hand side of (50) approaches to zero as $m\rightarrow\infty,$
and hence
\begin{equation}
(\Psi_{\gamma+t},e^{i\left\langle \gamma+\delta+t,x\right\rangle })=0,\text{
}\forall\delta\in\Gamma(k,p),\text{ }p\leq0,\text{ }\delta\neq0.
\end{equation}

Now we consider \textbf{Case }2 of Remark 2. Then $\delta$ has the form (44)
and $p>0.$ In this case we iterate (42) as follows. Isolate the terms in the
right-hand side of (42) containing the multiplicand $(\Psi_{\gamma
+t},e^{i\left\langle \gamma+t,x\right\rangle })$ which occurs in the case
$\gamma_{1}=\delta$ and use (42) for the other terms to get
\[
(\Psi_{\gamma+t},e^{i\left\langle \gamma+\delta+t,x\right\rangle }%
)=\frac{q_{\delta}}{d(\gamma,\delta)}(\Psi_{\gamma+t},e^{i\left\langle
\gamma+t,x\right\rangle })+\sum_{\gamma_{1},\gamma_{2}}\frac{q_{\gamma_{1}%
}q_{\gamma_{2}}(\Psi_{\gamma+t},e^{i\left\langle \gamma+\delta-\gamma
_{1}-\gamma_{2}+t,x\right\rangle })}{d(\gamma,\delta)d(\gamma,\delta
-\gamma_{1})}.
\]
Isolating again the terms containing the multiplicand $(\Psi_{\gamma
+t},e^{i\left\langle \gamma+t,x\right\rangle })$ which occurs in the case
$\gamma_{2}=\delta-\gamma_{1}$ and using again (42) for the other terms \ and
repeating this process $p-1$ times we obtain%
\begin{align}
(\Psi_{\gamma+t},e^{i\left\langle \gamma+\delta+t,x\right\rangle })  &
=c(\gamma,\delta)(\Psi_{\gamma+t},e^{i\left\langle \gamma+t,x\right\rangle
})+\\
&  \sum_{\gamma_{1},\gamma_{2},...,\gamma_{p}}\frac{q_{\gamma_{1}}%
q_{\gamma_{2}}...q_{\gamma_{p}}(\Psi_{\gamma+t},e^{i\left\langle \gamma
+\delta-\gamma(p)+t,x\right\rangle })}{d(\gamma,\delta)d(\gamma,\delta
-\gamma_{1})...d(\gamma,\delta-\gamma(p-1))},\nonumber
\end{align}
where $c(\gamma,\delta)$ is defined in (48) and (49), $\delta-\gamma(p)\neq0$
and by (46) $\delta-\gamma(p)\in\Gamma(k,s),$ $s\leq0.$ Therefore it follows
from (51) that the second term of the right-hand side of (52) is zero and
hence we have
\begin{equation}
(\Psi_{\gamma+t},e^{i\left\langle \gamma+\delta+t,x\right\rangle }%
)=c(\gamma,\delta)(\Psi_{\gamma+t},e^{i\left\langle \gamma+t,x\right\rangle
}).
\end{equation}

Since the systems $\left\{  e^{i\left\langle \gamma+t,x\right\rangle }%
:\gamma\in\Gamma\right\}  $ is the orthonormal basis in $L_{2}(F)$ and
$\left\Vert \Psi_{\gamma+t}\right\Vert \neq0$, formulas (41), (51) and (53)
imply that $(\Psi_{\gamma+t},e^{i\left\langle \gamma+t,x\right\rangle }%
)\neq0.$ Hence, there exists eigenfunction, denoted again by $\Psi_{\gamma
+t},$ satisfying (18). Now using (41), (51), (53) and the obvious relation
$\Gamma(k+)=\cup_{p\in\mathbb{N}}\Gamma(k,p)$ we get the proof of (47). Thus
we have proved that any eigenfunction normalized by (18) has the form (47). It
implies that there is only one eigenfunction normalized by (18) and hence the
right-hand sides of (19) and (47) are the same. The theorem is proved
\end{proof}

Now we consider two and three dimensional case.

\begin{theorem}
In the case $d=2$ and $d=3$ the results of Theorem 1 continue to hold if
condition (4) is replaced by $q\in L_{2}(F).$
\end{theorem}

\begin{proof}
If $q\in L_{2}(F),$ then we have
\begin{equation}
\sum_{\gamma\in\Gamma(k\pm)}\mid q_{\gamma}\mid^{2}<\infty.
\end{equation}
On the other hand, it is clear that if $d=2,3$ and (14) holds, then
\begin{equation}
\sum_{\gamma_{1}\in\Gamma(k+)}\left(  \frac{1}{\mid\gamma+t\mid^{2}-\mid
\gamma+\gamma_{1}+t\mid^{2}}\right)  ^{2}<\infty.
\end{equation}
The inequalities (54) and \ (55) and the Schwarz inequality for $l_{2}$ imply
that
\begin{equation}
\left\vert A(\gamma)e^{i\left\langle \gamma+t,x\right\rangle }\right\vert
<\infty,\text{ }\forall x.
\end{equation}
Moreover, using (34) one can easily verify that if (14) holds, then for fixed
$\gamma_{1},\gamma_{1},...,\gamma_{s-1}$ from $\Gamma(k+)$ the following
relation
\begin{equation}
\sum_{\gamma_{s}\in\Gamma(k+)}\left(  \frac{1}{\mid\gamma+t\mid^{2}-\mid
\gamma+\gamma_{1}+\gamma_{2}+...+\gamma_{s}+t\mid^{2}}\right)  ^{2}=O(s^{-1}).
\end{equation}
is satisfied. The relations (55)-(57) continue to hold if $\gamma$ is replaced
by $b_{j}$ for $j=1,2,...,s$ and (15)-(17) are satisfied. Therefore instead of
(4) and (34) using (54) and (57), respectively and repeating the proof of
Theorem 1 by using the Schwarz inequality for $l_{2},$ we get the proof of the theorem
\end{proof}

\section{On the Root Functions for the Multiple Eigenvalues}

In this chapter we consider the root function of $L_{t}(q)$ corresponding to
the multiple eigenvalues (15) and find the necessary and sufficient conditions
on the potential which provide some root functions to be eigenfunctions. For
this we introduce the following notation.

\begin{notation}
The integers $p_{1}\geq p_{2}\geq...$ defined in (16), in general, are not
different. There are $p,$ where $p\leq m,$ different numbers among them
denoted by $n_{1}>n_{2}>...>n_{p}.$ Then the vectors $b_{1},b_{2},...,b_{m}$
defined in (15) belong to $\Gamma(k,n_{1}),\Gamma(k,n_{2}),...$ and
$\Gamma(k,n_{p})$ (see (16) for the definition of $\Gamma(k,n_{l})$) which are
the points of the lattice $\Gamma$ lying on the parallel hyperplanes
$P(k,n_{1}),P(k,n_{2}),...$ and $P(k,n_{p})$ called as first, second and
$p$-th planes respectively, where $P(k,n_{l})=\left\{  x+n_{l}v_{k}:x\in
P(k)\right\}  $ and $P(k)$ is the hyperplane generated by the vectors
$v_{1},v_{2},...,v_{k-1},v_{k+1},...,v_{d}.$ We redenote the vectors
$b_{1},b_{2},...,b_{m}$ by $b_{l,j}$ for $j=1,2,...,s_{l}$ if $b_{l,j}%
\in\Gamma(k,n_{l}).$ Let $H(n_{l})$ and $H(n_{l},j)$ be respectively the
spaces spanned by%
\[
\left\{  e^{i\left\langle a+nv_{k}+t,x\right\rangle }:a\in\Gamma
(k),n>n_{l}\right\}  \And\left\{  e^{i\left\langle a+nv_{k}+t,x\right\rangle
}:a\in\Gamma(k),n>n_{l}\right\}  \cup\left\{  e^{i\left\langle b_{l,j}%
+t,x\right\rangle }\right\}  .
\]

\end{notation}

\begin{remark}
By Notation 1, $b_{l,j}$ belongs to $l$-th plane and hence has the form
\begin{equation}
b_{l,j}=a_{l,j}+n_{l}v_{k},\text{ }a_{l,j}\in\Gamma(k).
\end{equation}
Besides, comparing Notations 1 with the notations of (16) and (17) we see that
$s_{1}=s$ and $b_{i}\in\Gamma(k,n_{1})$ for $i=1,2,...,s,$ that is,
$b_{i}=a_{i}+n_{1}v_{k},$ $a_{i}\in\Gamma(k)$ for $i=1,2,...,s.$ Then by (15)
and (58) we have the equalities
\begin{equation}
\lambda=\mid a_{i}+n_{1}v_{k}+t\mid^{2}=\mid a_{l,j}+n_{l}v_{k}+t\mid^{2}%
\end{equation}
for all $l=1,2,...,p$ $;$ $j=1,2,...,s_{l}$ and $i=1,2,...,s$ that will be
used in the future.
\end{remark}

The following statements follows from the definitions of $\Gamma(k+)$,
$S(k+),$ $H(n_{l})$ and $H(n_{l},j)$ by using Theorem 1(a) and the approach of
the proof of Theorem 2.2 of [8]

\begin{theorem}
If $q\in S(k+),$ then the followings hold:

$(a)$ The operator $L_{t}(q)$ is invariant in $H(n_{l})$ and $H(n_{l},j)$ for
all $l$ and $j.$

$(b)$ For all $l$ and $\gamma\in\Gamma(k,n_{l})$ the number of linearly
independent root functions (eigenfunctions and associated functions) of the
operator $L_{t}(\varepsilon q)$ corresponding to the eigenvalue $\mid
\gamma+t\mid^{2}$ and lying in $H(n_{l})$ does not depend on $\varepsilon
\in\mathbb{C}.$ The statement continue to hold if $\Gamma(k,n_{l})$ and
$H(n_{l})$ is replaced by $\gamma\in\left(  \Gamma(k,n_{l})\cup\left\{
b_{j,l}\right\}  \right)  $ and $H(n_{l},j)$ respectively.

$(c)$ For all $l$ and $j$ the operator $L_{t}(q)$ has a root function of the
form%
\begin{equation}
\varphi_{l,j}(x)=e^{i\left\langle a_{l,j}+n_{l}v_{k}+t,x\right\rangle }+%
%TCIMACRO{\tsum \limits_{n=n_{l}+1}^{\infty}}%
%BeginExpansion
{\textstyle\sum\limits_{n=n_{l}+1}^{\infty}}
%EndExpansion
\left(
%TCIMACRO{\tsum \limits_{a\in\Gamma(k)}}%
%BeginExpansion
{\textstyle\sum\limits_{a\in\Gamma(k)}}
%EndExpansion
c(a,n)e^{i\left\langle a+nv_{k}+t,x\right\rangle }\right)
\end{equation}
where $c(a,n)=(\varphi_{l,j},e^{i\left\langle a+nv_{k}+t,x\right\rangle }).$
\end{theorem}

\begin{proof}
$(a)$ Let $Q$ be the subset of of the lattice $\Gamma$ such that $q_{\gamma
}\neq0$ if $\gamma\in Q.$ By (3) we have $Q\subset\Gamma(k+).$ Therefore if
$f\in H(n_{l},j)$, then $qf\in H(n_{l})\subset H(n_{l},j),$ that yields $(a).$

$(b)$ Now from $(a)$ and Theorem 1(a) arguing as in the proof of Theorem 2.2
of [8] we get the proof of $(b).$ Namely, we argue as follows. Let $D$ be
small disk with the center $\mid\gamma+t\mid^{2},$ where $\gamma\in
\Gamma(k,n_{l}),$ and containing no other eigenvalue of $L_{t}(0).$ By Theorem
1(a) $D$ contains no eigenvalue of $L_{t}(\varepsilon q)$ except $\mid
\gamma+t\mid^{2}.$ Therefore the projection of $L_{t}(\varepsilon q)$ defined
by contour integration over the boundary of $D$ depend continuously on
$\varepsilon$ which imply the proof of $(b)$ for the space $H(n_{l}).$ The
proof for $H(n_{l},j)$ is the same.

$(c)$ By Notation 1 the operator $L_{t}(0)$ in $H(n_{l})$ and $H(n_{l},j)$ has
respectively $n$ and $n+1$ linearly independent eigenfunctions corresponding
to the eigenvalue $\lambda,$ where $n=s_{1}+s_{2}+...+s_{l}.$ Hence, it
follows from $(b)$ that the operator $L_{t}(q)$ has a root function $\varphi$
such that $\varphi\in H(n_{l},j)$ but $\varphi\notin H(n_{l}),$ where
$H(n_{l},j)$ is the orthogonal sum of $H(n_{l})$ and the one dimensional space
generated by $e^{i\left\langle b_{j,l}+t,x\right\rangle }.$ It mean that the
projection of $\varphi$ onto the one dimensional space is nonzero, that is,
$(\varphi,e^{i\left\langle b_{l,j}+t,x\right\rangle })$ is not zero.
Therefore, without loss of generality, the latter number can be assumed to be
$1$.
\end{proof}

Now we find the necessary and sufficient condition on the potential $q$ for
which $\varphi_{l,j}$ is an eigenfunction. We will say that $\varphi$ is the
$l$-th associated function of $L_{t}(q)$ if%
\[
\left(  L_{t}(q)-\lambda I\right)  ^{l}\varphi\neq0\text{ }\And\left(
L_{t}(q)-\lambda I\right)  ^{l+1}\varphi=0.
\]
In other words, $\varphi$ is called the first associated functions if
\begin{equation}
\left(  L_{t}(q)-\lambda I\right)  \varphi=\Psi
\end{equation}
and $\Psi$ is an eigenfunction. If (61) holds and $\Psi$ is an $(l-1)$-th
associated function then we say that $\varphi$ is an $l$-th associated function.

\begin{theorem}
$(a)$ The functions $\varphi_{1,j}$ for $j=1,2,...s$ defined in (60) are the
eigenfunctions of the operator $L_{t}(q).$

$(b)$ If $l>1,$ then the functions $\varphi_{l,j}$ for $j=1,2,...s_{l}$ are
either the eigenfunctions or the $n$-th associated functions of $L_{t}(q),$
where $n<l$
\end{theorem}

\begin{proof}
Statement $(a)$ follows from Theorem 3. To prove $(b)$ let us first consider
the case $l=2.$ Using the relations $\varphi_{2,j}\in H(n_{2},j)$ and $q\in
S(k+)$ one can readily see that $\left(  -\Delta+\left(  q-\lambda\right)
I\right)  \varphi_{2,j}\in H(n_{2})$. On the other hand, it follows from $(a)$
that $H(n_{2})$ does not contain \ the associated functions. Therefore
$\left(  -\Delta+\left(  q-\lambda\right)  I\right)  \varphi_{2,j}$ is either
eigenfunction or zero. It means that $\varphi_{2,j}$ is either the
eigenfunction or the first associated functions. Similarly $\left(
-\Delta+\left(  q-\lambda\right)  I\right)  \varphi_{l,j}\in H(n_{l})$.
Therefore by using the induction method we obtain that $\varphi_{l,j}$ is
either eigenfunction or the $n$-th associated functions, where $n<l$
\end{proof}

Now we will consider in detail the root functions $\varphi_{2,j}$
corresponding to the vectors of the second plane. By Theorem 7$(b)$ the root
functions corresponding to the vectors of the second plane are either
eigenfunctions or the first associated functions. For the simplicity of the
notations we omit the indices and denote by $\gamma$ the arbitrary vector from
$\left\{  b_{2,j}:j=1,2,...,s_{2}\right\}  $ and by $\varphi$ the
corresponding root function. Thus $\gamma=\delta+n_{2}v_{k},$ $\delta\in
\Gamma(k)$ and $\varphi$ is an eigenfunction corresponding to the eigenvalue
$\lambda=\mid\delta+n_{2}v_{k}+t\mid^{2}$ if and only if the following
equality holds
\begin{equation}
(\Delta+\lambda)\varphi=q\varphi,
\end{equation}
where $\lambda$ is defined in (15) and (59). Using the decompositions%
\[
\varphi(x)=e^{i\left\langle \delta+n_{2}v_{k}+t,x\right\rangle }+%
%TCIMACRO{\tsum \limits_{n=n_{2}+1}^{\infty}}%
%BeginExpansion
{\textstyle\sum\limits_{n=n_{2}+1}^{\infty}}
%EndExpansion
\left(
%TCIMACRO{\tsum \limits_{a\in\Gamma(k)}}%
%BeginExpansion
{\textstyle\sum\limits_{a\in\Gamma(k)}}
%EndExpansion
c(a,n)e^{i\left\langle a+nv_{k}+t,x\right\rangle }\right)
\]
(see (60)) and taking into account that $q\in\Gamma(k+),$ that is,%
\[
q(x)=%
%TCIMACRO{\tsum \limits_{m=1}^{\infty}}%
%BeginExpansion
{\textstyle\sum\limits_{m=1}^{\infty}}
%EndExpansion
\sum_{u\in\Gamma(k)}q_{u+mv_{k}}e^{i\left\langle u+mv_{k},x\right\rangle }%
\]
we see that (62) holds if and only if the following system of equations holds%
\begin{equation}
(\lambda-\mid a+nv_{k}+t\mid^{2})c(a,n)=q_{a-\delta+\left(  n-n_{2}\right)
v_{k}}+%
%TCIMACRO{\tsum \limits_{m=1}^{\infty}}%
%BeginExpansion
{\textstyle\sum\limits_{m=1}^{\infty}}
%EndExpansion
\left(
%TCIMACRO{\tsum \limits_{u\in\Gamma(k)}}%
%BeginExpansion
{\textstyle\sum\limits_{u\in\Gamma(k)}}
%EndExpansion
c(a-u,n-m)q_{u+mv_{k}}\right)
\end{equation}
for all $n\geq n_{2}+1$ and $a\in\Gamma(k).$ Now in (63) instead of $n$ taking
$n_{2}+1,n_{2}+2,....,n_{1}-1,n_{1}$ and taking into account that
$c(a-u,n-m)=0$ for $n-m<n_{2}+1$ we obtain%

\begin{align}
(\lambda- &  \mid a+\left(  n_{2}+1\right)  v_{k}+t\mid^{2})c(a,n_{2}%
+1)=q_{a-\delta+v_{k}},\nonumber\\
(\lambda- &  \mid a+\left(  n_{2}+2\right)  v_{k}+t\mid^{2})c(a,n_{2}%
+2)=q_{a-\delta+2v_{k}}+%
%TCIMACRO{\tsum \limits_{u\in\Gamma(k)}}%
%BeginExpansion
{\textstyle\sum\limits_{u\in\Gamma(k)}}
%EndExpansion
c(a-u,n_{2}+1)q_{u+v_{k}},\nonumber\\
&  \cdot\cdot\cdot\cdot\cdot\cdot\cdot\cdot\cdot\cdot\cdot\cdot\cdot\cdot
\cdot\cdot\cdot\cdot\cdot\cdot\\
(\lambda- &  \mid a+\left(  n_{1}-1\right)  v_{k}+t\mid^{2})c(a,n_{1}%
-1)=q_{a-\delta+\left(  n_{1}-1-n_{2}\right)  v_{k}}+\nonumber\\
&
%TCIMACRO{\tsum \limits_{m=1}^{n_{1}-n_{2}-2}}%
%BeginExpansion
{\textstyle\sum\limits_{m=1}^{n_{1}-n_{2}-2}}
%EndExpansion%
%TCIMACRO{\tsum \limits_{u\in\Gamma(k)}}%
%BeginExpansion
{\textstyle\sum\limits_{u\in\Gamma(k)}}
%EndExpansion
c(a-u,n_{1}-1-m)q_{u+mv_{k}},\nonumber\\
(\lambda- &  \mid a+n_{1}v_{k}+t\mid^{2})c(a,n_{1})=q_{a-\delta+\left(
n_{1}-n_{2}\right)  v_{k}}\nonumber\\
&  +%
%TCIMACRO{\tsum \limits_{m=1}^{n_{1}-n_{2}-1}}%
%BeginExpansion
{\textstyle\sum\limits_{m=1}^{n_{1}-n_{2}-1}}
%EndExpansion%
%TCIMACRO{\tsum \limits_{u\in\Gamma(k)}}%
%BeginExpansion
{\textstyle\sum\limits_{u\in\Gamma(k)}}
%EndExpansion
c(a-u,n_{1}-m)q_{u+mv_{k}}.\nonumber
\end{align}
From the first equation we express $c(a,n_{2}+1)$ for $a\in\Gamma(k)$ in term
of the Fourier coefficients of the potential. In the right-hand side of the
second equation only $c(a,n_{2}+1)$ for $a\in\Gamma(k)$ takes part. Therefore
from the second equation we express $c(a,n_{2}+2)$ for $a\in\Gamma(k)$ in term
of the Fourier coefficients of the potential. In this way, from the third,
forth.... and $\left(  n_{1}-n_{2}-1\right)  $-th equations we express
$c(a,n_{2}+3)$ $c(a,n_{2}+4),...$ and $c(a,n_{1}-1)$ for $a\in\Gamma(k)$ in
term of the Fourier coefficients of the potential. Thus the right hand side of
the last equation for fixed $u$ is a polynomial of the Fourier coefficients.
On the other hand, if $a=a_{i}$ for $i=1,2,...,s,$ then by (59) the left hand
side of the last equation of (64) is zero. Therefore we obtain the following
$s$ equalities%
\begin{equation}
q_{a_{i}-\delta+\left(  n_{1}-n_{2}\right)  v_{k}}+%
%TCIMACRO{\tsum \limits_{m=1}^{n_{1}-n_{2}-1}}%
%BeginExpansion
{\textstyle\sum\limits_{m=1}^{n_{1}-n_{2}-1}}
%EndExpansion
\left(
%TCIMACRO{\tsum \limits_{u\in\Gamma(k)}}%
%BeginExpansion
{\textstyle\sum\limits_{u\in\Gamma(k)}}
%EndExpansion
c(a_{i}-u,n_{1}-m)q_{u+mv_{k}}\right)  =0,\text{ }\forall i=1,2,...,s,
\end{equation}
where $c(a_{i}-u,n_{1}-m)$ is explicitly expressed by the Fourier coefficients
of the potential $q.$ The equalities (65) can be written in the form%
\begin{equation}%
%TCIMACRO{\tsum \limits_{u\in\Gamma(k)}}%
%BeginExpansion
{\textstyle\sum\limits_{u\in\Gamma(k)}}
%EndExpansion
Q(b_{i},u)=0,\text{ }\forall i=1,2,...,s,
\end{equation}
\ where $Q(b_{i},u)$ is the polynomial of the Fourier coefficient of the
potential $q$ and $b_{i}=a_{i}+n_{1}v_{k}$ (see Remark3) for $i=1,2,...,s$ are
the vectors defined by (15) and lying on the first plane $P(k,n_{1})$. Thus we have

\begin{theorem}
The function $\varphi_{2,j},$ defined in (60), is an eigenfunction if and only
if (66) holds.
\end{theorem}

\begin{proof}
We have proved that if $\varphi_{2,j}$ is an eigenfunction, then (66). holds.
Now suppose that $\varphi_{2,j}$ is not eigenfunction. Then, by Theorem 7, it
is first associated function and hence
\begin{equation}
(\Delta+\mid b_{2,j}+t\mid^{2})\varphi_{2,j}=q\varphi_{2,l}+\Psi,
\end{equation}
where $\Psi\in span\{$ $\varphi_{1,j}:j=1,2,...,s\}.$ It means that
\[
\Psi=%
%TCIMACRO{\tsum \limits_{i\in E}}%
%BeginExpansion
{\textstyle\sum\limits_{i\in E}}
%EndExpansion
c_{i}\varphi_{1,i},
\]
where $E$ is an nonempty subset of $\{$ $1,2,...,s\}$ and $c_{i}\neq0$ for all
$i\in E.$ Instead of (62) using (67) arguing as in the proof of (64) and
taking into account that $\varphi_{1,i}\in H(n_{1},i)$, we get the system of
equations whose first $n_{1}-n_{2}-1$ equations coincides with the first
$n_{1}-n_{2}-1$ equations of (64) and the last equation has the form%
\begin{align}
(\lambda- &  \mid a+n_{1}v_{k}+t\mid^{2})c(a,n_{1})=q_{a-\delta+\left(
n_{1}-n_{2}\right)  v_{k}}+\nonumber\\
&
%TCIMACRO{\tsum \limits_{m=1}^{n_{1}-n_{2}-1}}%
%BeginExpansion
{\textstyle\sum\limits_{m=1}^{n_{1}-n_{2}-1}}
%EndExpansion%
%TCIMACRO{\tsum \limits_{u\in\Gamma(k)}}%
%BeginExpansion
{\textstyle\sum\limits_{u\in\Gamma(k)}}
%EndExpansion
c(a-u,n-m)q_{u+mv_{k}}+(\Psi,e^{i(a+n_{1}v_{k}}).
\end{align}
As in the case (64) the terms $c(a-u,n-m)$ in the right hand side of (68) are
obtained from the first $n_{1}-n_{2}-1$ equations of (64). Therefore arguing
as in the proof of (65) we get
\[
q_{a_{i}-\delta+\left(  n_{1}-n_{2}\right)  v_{k}}+%
%TCIMACRO{\tsum \limits_{m=1}^{n_{1}-n_{2}-1}}%
%BeginExpansion
{\textstyle\sum\limits_{m=1}^{n_{1}-n_{2}-1}}
%EndExpansion%
%TCIMACRO{\tsum \limits_{u\in\Gamma(k)}}%
%BeginExpansion
{\textstyle\sum\limits_{u\in\Gamma(k)}}
%EndExpansion
c(a_{i}-u,n_{1}-m)q_{u+mv_{k}}=c_{i}\neq0,\text{ }\forall i\in E,
\]
and hence equality (66) does not hold for $i\in E.$
\end{proof}

The investigation of the root functions $\varphi_{l,j}$ corresponding to the
vectors of the $l$-th plane for $l>2$ can be considered in the similar way.
However, the multiplicity of the large eigenvalues of the multidimensional
Schrodinder operators $L_{t}(q)$ is very large number. For example if the
period lattice of the potential $q$ is $2\pi\mathbb{Z}^{d}$ then the
multiplicity of the eigenvalue $\left\vert \gamma\right\vert ^{2},$ where
$\gamma\in\mathbb{Z}^{d},$ is of order $\left\vert \gamma\right\vert ^{d-2}$
\ and hence approaches infinity as $\left\vert \gamma\right\vert
\rightarrow\infty.$ Therefore the detailed investigation of the root functions
$\varphi_{l,j}$ for all $l$ is technically very complicated. In order to avoid
eclipsing the essence of this paper by the technical details we consider only
the case $l=2$. Moreover, in the one-dimensional case there are only two root
function corresponding to the double eigenvalues $\left(  2\pi n\right)  ^{2}$
($n\neq0)$ and $\left(  2\pi n+\pi\right)  ^{2}$ and hence we do not need to
consider the root functions corresponding to the vectors lying in $l$-th
planes for $l>2.$ The one dimensional case for the potential $q$ from
$L_{2}[0,1]$ is investigated in [7,8]. The case $q(x)=Ae^{2\pi irx},$ where
$A\in\mathbb{C}$ and $r\in\mathbb{Z}$, was investigated in detail in [9].

Now we consider the more general case $q\in L_{1}[0,1]$ by using some formulas
obtained in the papers [12,13]. In [12] we investigated the one dimensional
operators $L_{t}(q)$ for $t\neq0,\pi$ corresponding to the boundary value
problems (12) for the potential $q$ satisfying (11). Let us consider the case
$t=0.$ The case $t=\pi$ is similar. Since for $q\in L_{1}[0,1]$ Fourier
decomposition (9) does not hold we can not immediately use Theorem 8 and
Theorem 4.2 of the paper [8]. Therefore we consider this case by using the
results of the papers [12, 13]. Namely, we use the followings. In Theorem 1 of
[12] we proved that $\left(  2\pi n\right)  ^{2}$ for $n\neq0$ is the double
eigenvalue of $L_{0}(q).$ Instead of the decomposition (9) we use the formula
\begin{equation}
((2\pi n)^{2}-(2\pi(n+p))^{2})(\Psi,e^{i2\pi(n+p)x})=\sum_{m\in\mathbb{N}%
}q_{m}(\Psi,e^{i2\pi(n+p-m)x})
\end{equation}
which was proved in [13,12] (see Lemma 1 of [13] and formula (16) of [12]),
where $\Psi$ is an eigenfunction corresponding to the eigenvalue $(2\pi
n)^{2}$.

\begin{theorem}
Suppose (11) holds. Let $\Psi$ be an eigenfunction of $L_{0}(q)$ corresponding
to the eigenvalue $\left(  2\pi n\right)  ^{2},$ where $n\in\mathbb{N}$.

$(a)$ If the first of the numbers
\begin{equation}
(\Psi,e^{-i2\pi nx}),\text{ }(\Psi,e^{i2\pi nx})
\end{equation}
is not zero, then $\Psi$ has the form \textit{ }%
\begin{equation}
\text{ }e^{-i2\pi nx}+\sum_{p\in\mathbb{N}}c_{p}e^{i2\pi(-n+p)x}%
\end{equation}
If the first of the numbers in (70) is zero and the second is not zero, then
$\Psi$ has the form \textit{ }%
\begin{equation}
\text{ }e^{i2\pi nx}+\sum_{p\in\mathbb{N}}d_{p}e^{i2\pi(n+p)x}%
\end{equation}
If both of the numbers in (70) are zero then $\Psi$ is the zero function

$(b)$ If the geometric multiplicity of the eigenvalue $(2\pi n)^{2}$ is two,
then there exist linearly independent eigenfunctions $\Psi_{-n}$ and $\Psi
_{n}$ having the forms (71) and (72) respectively.

$(c)$ \ If the geometric multiplicity of the eigenvalue $(2\pi n)^{2}$ is one,
then the operator $L_{0}(q)$ has an eigenfunction $\Psi$ of the form (72) and
an associated function $\Phi$ corresponding to the eigenfunction $c\Psi,$
where $c$ is a nonzero constant, of the form (71).
\end{theorem}

\begin{proof}
$(a)$ Iterating (69) as was done in the proof of Theorem 2 in [12], we see
that
\begin{equation}
(\Psi,e^{i2\pi(n+p)x})=0
\end{equation}
for all $p<-2n.$ Therefore if the first of the numbers in (70) is not zero
then $\Psi$ has the form (71). In the same way we conclude that if the first
of the numbers in (70) is zero and the second is not zero then (73) holds for
$p<0$ and $\Psi$ has the form (72). These argument also shows that if both of
the numbers in (70) are zero then (73) holds for all $p\in\mathbb{Z}$.

$(b)$ If both linearly independent eigenfunctions has the form (71), then some
multiple of their difference has the form (72), since there exist three cases:
case (71), case (72) and zero function. If both linearly independent
eigenfunctions has the form (72), then their difference is zero that
contradicts to the independence.

$(c)$ By the definition of the associated function we have $\left(
L_{0}(q)-(2\pi n)^{2}\right)  \Phi=c\Psi,$ where $c$ is a nonzero number.
Multiplying both sides by $e^{i2\pi(n+p)x}$ and then arguing as in the proof
of (69) we obtain
\begin{equation}
(2\pi n)^{2}-(2\pi(n+p))^{2})(\Phi,e^{i2\pi(n+p)x})=\sum_{m\in\mathbb{N}}%
q_{m}(\Phi,e^{i2\pi(n+p-m)x})+c(\Psi,e^{i2\pi(n+p)x})
\end{equation}
If $p<-2n$ then as we noted in $(a)$ the second term in the right side of (74)
is zero. Therefore iterating (74) we obtain that
\begin{equation}
(\Phi,e^{i2\pi(n+p)x})=0
\end{equation}
for all $p<-2n.$ Now let $p=-2n.$ Then the left side of (74) is zero and by
(75) the first term in the right side of (74) is also zero. It implies that
$(\Psi,e^{-i2\pi nx})=0$ and hence by (a) $\Psi$ has the form (72). It remains
to show that $\Phi$ has the form (71), that is, (75) is not true for $p=-2n.$
If (75) is true for $p=-2n,$ then arguing as above we obtain that (75) holds
for $p<0.$ It with (74) for $p=0$ implies that $(\Psi,e^{i2\pi nx})=0$ which
is a contradiction
\end{proof}

Let (71) be the eigenfunction $\Psi$. Then in (69) replacing $p$ and $\Psi$ by
$0$ and (71) and taking into account that $q_{m}=0$ for $m\leq0$ we obtain%
\begin{equation}
q_{2n}+%
%TCIMACRO{\tsum \limits_{p=1}^{2n-1}}%
%BeginExpansion
{\textstyle\sum\limits_{p=1}^{2n-1}}
%EndExpansion
c_{p}q_{2n-p}=0,
\end{equation}
where $c_{p}=(\Psi,e^{i2\pi(-n+p)x}).$ In (69) replacing $n$ with $-n$ and
then using the equalities $c_{0}=1$ and $c_{p}=0$ for $p<0$ we obtain
\[
((2\pi n)^{2}-(2\pi(-n+p))^{2})c_{p}=q_{1}c_{p-1}+q_{2}c_{p-2}+...+q_{p-1}%
c_{1}+q_{p}%
\]
Iterating it we get
\begin{equation}
c_{p}=\frac{1}{4\pi^{2}p(2n-p)}\left(  q_{p}+\sum_{k=1}^{p-1}\sum_{n_{1}%
,n_{2},...,n_{k}}\frac{q_{n_{1}}q_{n_{2}}...q_{n_{k}}q_{p-n(k)}}%
{b(n,p,k)}\right)  ,
\end{equation}
where $n(k)=n_{1}+n_{2}+...+n_{k},$ $\{n_{1},n_{2},...,n_{k},p-n(k)\}\subset
\mathbb{N}$ and
\[
b(n,p,k)=%
%TCIMACRO{\tprod \limits_{s=1}^{k}}%
%BeginExpansion
{\textstyle\prod\limits_{s=1}^{k}}
%EndExpansion
(4\pi^{2}(p-n(s))(2n-p+n(s))).
\]
Therefore the equalities (76) and (77) give us the equality
\begin{equation}
q_{2n}+%
%TCIMACRO{\tsum \limits_{p=1}^{2n-1}}%
%BeginExpansion
{\textstyle\sum\limits_{p=1}^{2n-1}}
%EndExpansion
\frac{q_{2n-p}}{4\pi^{2}p(2n-p)}\left(  q_{p}+\sum_{k=1}^{p-1}\sum
_{n_{1},n_{2}...,n_{k}}\frac{q_{n_{1}}q_{n_{1}}...q_{n_{k}}q_{p-n(k)}%
}{b(n,p,k)}\right)  =0.
\end{equation}

Now let (71) and (72) be the associated function $\Phi$ and eigenfunction
$\Psi$ respectively. Instead of (69) using (74) and repeating the above
argument by taking into account that $(\Psi,e^{i2\pi nx})=1$ and
$(\Psi,e^{i2\pi(n+p)x})=0$ for $p<0,$ we get the equality obtained from (78)
by replacing $0$ with $-c,$ where $c\neq0$. Thus we have proved the following.

\begin{theorem}
Suppose the conditions in (11) hold. Then the geometric multiplicity of the
eigenvalue $\left(  2\pi n\right)  ^{2}$ for $n\neq0$ of the operator
$L_{0}(q)$ is two if and only if (78) holds. The similar result holds for the
eigenvalue $\left(  2\pi n+\pi\right)  ^{2}$ of $L_{\pi}(q).$
\end{theorem}

\end{document}